\documentclass[11pt,letterpaper]{amsart}
\usepackage{epsfig,layout,graphics}
\usepackage{pgf,pgfarrows,pgfnodes,pgfautomata,pgfheaps}
\usepackage{epstopdf}
\usepackage{multicol}
\usepackage{amsmath,amssymb}
\usepackage{mathtools}
\usepackage{multicol}

 \theoremstyle{plain}
 \newtheorem{thm}{Theorem}[section]
 \numberwithin{equation}{section} 
 \numberwithin{figure}{section} 
 \theoremstyle{plain}
 
 \theoremstyle{plain}
 \newtheorem{algorithm}[thm]{Algorithm} 
 \theoremstyle{plain}
 \newtheorem{theorem}[thm]{Theorem}
 \theoremstyle{plain}
 
\theoremstyle{plain}
 \newtheorem{remark}[thm]{Remark}
 \theoremstyle{plain}
 
\def\P{{\mathcal{P}}}

\def\T{{\mathcal{T}}}
\def\M{{\mathcal{M}}}

\def\Forall{\quad \hbox{ for all }}
\def\curl{\operatorname{curl}}
\def\div{\operatorname{div}}

\def\<{{\langle}}
\def\>{{\rangle}}

\def\bu{{\bf u}}
\def\bv{{\bf v}}

\def\bVh{{\mathbf V_h}}
\def\bvh{{\mathbf v_h}}

\def\bv{{\mathbf v}}
\def\bu{{\mathbf u}}
\def\bw{{\mathbf w}}

\def\bh{{\mathbf h}}

\def\bbf{{\mathbf f}}

\def\bV{{\mathbf V}}

\def\du#1#2#3{\overset{#3}{\underset{#2}{#1}}}
\def\curl{\operatorname{curl}}
\def\div{\operatorname{div}}
 
\begin{document}
\bibliographystyle{plain}  
\title[Cascadic Multilevel  Algorithms for Solving Symmetric SPS]
{Cascadic Multilevel  Algorithms for Symmetric Saddle Point Systems}

\author{Constantin Bacuta}
\address{University of Delaware,
Department of Mathematics,
501 Ewing Hall 19716}
\email{bacuta@math.udel.edu}



\begin{abstract} 
In this paper, we introduce a  multilevel   algorithm for approximating variational formulations of symmetric saddle point systems. The algorithm is  based on availability of  families  of  stable  finite element pairs and  on the availability of fast and accurate solvers for {\it symmetric positive definite systems}. 
On each  fixed level  an efficient solver such as the gradient  or the conjugate gradient algorithm for inverting a Schur complement  is implemented.   The level change criterion follows the cascade principle and requires that   the iteration error be close to the expected discretization error. We prove  new estimates   that  relate the  iteration error and the    residual for the constraint equation. The new estimates  are the key ingredients  in  imposing  an efficient  level change criterion. The first iteration on each new level uses information about  the best approximation of the discrete solution from the previous level. The theoretical results and experiments show that the algorithms achieve optimal or close to optimal  approximation rates by performing a  non-increasing  number of  iterations on each level. 
Even though numerical results supporting the efficiency of the algorithms are presented for the Stokes system, the algorithms can be applied to a large class of boundary value problems, including first order systems  that can be reformulated at the continuous level as  symmetric saddle point problems, such as the  Maxwell equations. 
\end{abstract}
\maketitle

\section{Introduction}\label{section: introduction}

The  {\it cascade principle}  for  elliptic partial differential equations (PDEs)  was introduced  by  Deuflhard, Leinen and Yserentant in \cite{DeuflhardLeinenYserentant}. The main advantage of cascadic methods is that the iteration on each level is terminated as soon as the algebraic error is below the truncation or discretization error.  Shaidurov \cite{Shaidurov96} introduced a cascadic conjugate gradient and proved optimality in the energy norm for elliptic problems in two dimensions. The results were extended  by Bornemmann  and Deuflhard to the three dimensional problem in \cite{BornemannDeuflhard96}. 

In this paper we adopt the {\it cascade principle} in the context of  multilevel discretization of symmetric and coercive   saddle point  (SP)  systems. We let $ \bV$ and $Q$  be Hilbert  spaces,  and assume that $a(\cdot, \cdot) $ is a symmetric bounded and coercive bilinear form on  $\bV\times \bV$ that defines also the inner product on $\bV$, and that $b(\cdot, \cdot) $ is a continuous bounded bilinear form on $ \bV \times Q$  satisfying a continuous (LBB) or inf-sup condition. We denote  the  inner product  on $Q$ by $(\cdot , \cdot)$, and assume that the data ${\bf f}$ and $g$ belong to the dual spaces  $\bV^*$ and  $Q^*$, respectively.  We consider the variational problem: Find $(\bu, p)  \in \bV \times Q$ such that

\begin{equation}\label{abstract:variational}
\begin{array}{lcll}
a(\bu,\bv) & + & b( \bv, p) &= <{\bf f},\bv>, \ \ \Forall \bv \in \bV,\\
b(\bu,q) & & & =<g,q> , \  \  \Forall q \in Q
\end{array} 
\end{equation}

There is a broad literature on the multilevel finite element discretization for \eqref{abstract:variational},  see  \cite{B06, B09, BrambleZhang, uzawa1, uzawa2, uzawa3, bpvscale, elman-golub, benzi-golub-liesen, vassilevski-wang,  XuThesis, XuSiamReview}.   More recent work in multilvel  approximation of  variational formulations of  saddle point type systems  can be found   in  \cite{B09, bansch-morin-nochetto, dahlke-dahmen-urban, kondratyuk-stevenson}. 
A {\it cascadic approach} for discretizing \eqref{abstract:variational} was done 
by Braess, Dahmen and Sarazin  for  the Stokes systems in   \cite{BraessSarazin95, BraessDahmen99}.

In this paper, we present a  general   {\it Cascadic Multilevel}   (CM) algorithm  for solving the problem \eqref{abstract:variational}.  We start by   assuming that  a  sequence of pairs $\{({\bV_k}, {\M_k})\}_{k \geq 1}$  that satisfies  a discrete $inf-sup$ condition for every $k\geq 1$, and a sequence of  {prolongation} operators  ${\P_{k,k+1}} : {\M_k} \to {\M_{k+1}}$   are available. 

\begin{algorithm} CM Algorithm\label{alg:CM} \ \ \\
\begin{itemize}
 \item  {\bf Set} { $j=1,  k=1$}, \  $ \textcolor{magenta}{\bu_0 =0\in \bV_{1}} $,  and  $\textcolor{blue}{p_0\in \M_1}$.
 \smallskip
\item {\bf Step CM1}:  Solve for $\textcolor{magenta}{\bu_j  \in  \bV_k}$  and $ \textcolor{blue}{ q_j \in \M_k}$:  
\[
\begin{aligned}
&  a(\textcolor{magenta}{\bu_j}, \bv) & =  &<f,\bv> - b(\bv, p_{j-1}), &\Forall & \bv  \in \textcolor{magenta}{\bV_{k}},\\
\smallskip
&( {\textcolor{blue}{q_j}},q)  & = & b(\textcolor{magenta}{\bu_j} ,q) -<g, q>, &\Forall &  \textcolor{blue}{ q\in \M_k}. 
 \end{aligned}
 \]
\item {\bf Step CM2}:  Compute $(\textcolor{magenta}{\bu_{j+1}},  \textcolor{blue}{ p_j})$ from by  $(\textcolor{magenta}{\bu_{j}},  \textcolor{blue}{ p_{j-1}})$ by a \\  \textcolor{red}{process} on  $(\textcolor{magenta}{\bV_k}, 
\textcolor{blue}{\M_k})$.\\
 \smallskip 
 
\item   Check a {\bf level  change}  condition \textcolor{magenta} {\bf (LC)}.\\  
 \smallskip
 
\item {\bf Step CM3}:  \textcolor{red}{Repeat} {\bf  CM2} with  \textcolor{red}{$j \to j+1$}  \\  \textcolor{blue}{until}  \textcolor{magenta} {\bf (LC)}   \textcolor{red}{ is satisfied}.
\smallskip

\item {} {\it Define}   $ \textcolor{blue}{p^{(k+1)}_{0}}:=  \textcolor{blue}{\P_{k,k+1}} (\textcolor{blue}{p_{j}})$,   {\it increase the level  $(k \to k+1)$}. \\
Increase   \textcolor{red}{$j \to j+1$}  and  {\bf Go To} {\bf  CM1} with  $p_{j-1}=\textcolor{blue}{p^{(k+1)}_{0}}$. 
  
\end{itemize}
\end{algorithm}
 The algorithm is quite general, and  we will consider  the cases when   the process of   {\bf Step CM2} is executed  by   Uzawa (U), Uzawa Gradient  (UG)  or Uzawa Conjugate Gradient  (UCG) ``one step''  iteration. This means that our  proposed CM  algorithm is a Schur complement type  process. The main computational challenge  for a typical {\it  one step}  iteration is to invert  the discrete operator  $A_k$  associated with  the form $a(\cdot, \cdot)$ on $\M_k$. 
 When a fixed  level iteration ends due to the level change criterion (LC), only $p_{j}$ needs to be prolongated to the next  space  $\M_{k+1}$. Thus, the CM algorithm  is a  {\it simple to implement}  iterative process. On the other hand,   this algorithm  is build on the premise  that the action of  $A_k^{-1}$ is fast and exact.

Some other multilevel  approaches for solving \eqref{abstract:variational} that are related with the proposed CM algorithm are as follows. In \cite{verfurtth84}, Verf{\"u}rth uses an inexact conjugate gradient algorithm on  a single fine level where the inexact elliptic  process is provided by a multigrid algorithm that requires a multilevel structure. Level wise, our  CM algorithm  moves always upwards and an exact elliptic  solver is called  at each iteration. 
In \cite{BraessSarazin95},  Braess and Sarazin, develop a multigrid algorithm for discretizing the Stokes system that 
is based on a smoother acting on the residual of the global system. In \cite{BraessDahmen99},  Braess and Dahmen provide sharp estimates of a cascadic approach for the Stokes problem  that uses  the smoothing procedure proposed in  \cite{BraessSarazin95}.  
We emphasize that, according to the terminology of  
\cite{BraessSarazin95},  the  Braess-Sarazin-Dahmen approach is $\bu$-dominated, i.e., $(\bu_{j+1}, p_{j})$, mainly depends on $\bu_j$, while  our proposed CM algorithm  is $p$ dominated, i.e., $(\bu_{j+1}, p_{j})$ mainly  depends  on  $p_{j-1}$. 

In \cite{BacShu13, BM12}, we investigated similar  multilevel algorithms based on the inexact Uzawa algorithms at the continuous level and on inexact processes for approximating continuous residuals.  When  the inexact process  acting on residuals  of the first equation is a standard  Galerkin projection, the algorithms proposed in \cite{BacShu13, BM12} become  particular versions  of the proposed CM algorithm. Nevertheless,  the level change criterion we propose, and the choice of stable families   of approximation spaces we use  in this paper,  lead to a different  type of  CM algorithm.  

One  novelty  of the  CM  algorithm  we propose in Section \ref{section:CM},  is the level change condition that takes full advantage of  the maximum expected  order of the discretization error. We find an iteration error estimator  that is easy to compute and works well with all tree (U, UG and UCG)  choices of iterative processes.   
For the non-convex domains, where the full regularity of the solution might be lost,  we consider special discrete spaces based on  graded meshes, and using the appropriate level change condition we are  able to recover optimal or close to optimal rates of approximation for the continuous solution. 

 The rest of the  paper is organized as follows. In Section \ref{section:notation}, we introduce the needed notation for building the theory and the analysis for the  CM algorithm. In Section \ref{section:UGC}, we review  the Uzawa, UG and the UCG algorithms and find a  sharp error  estimator for the iteration error.   
 In Section \ref{section:CM}, we  specify  the (LC) condition and concrete  level solvers in order to define  implementable CM algorithms.  In Section \ref{section:NR-Stokes}, we present  the performance of a few  versions of CM algorithms  for different choices of solvers and discrete spaces for approximating the solution of the Stokes system. We summarize our conclusions in Section \ref{section:conclusion}. 
 
\section{General Framework and  Notation }\label{section:notation}

We consider the standard notation for the saddle point problem (SPP) abstract framework. We  let  $\bV$ and  $Q$ be two Hilbert spaces with  inner products $a(\cdot, \cdot)$ and 
$(\cdot, \cdot)$ respectively,  with the corresponding induced norms $|\cdot|_{\bV} =|\cdot| =a(\cdot, \cdot)^{1/2}$ 
and $\|\cdot\|_Q=\|\cdot\| =(\cdot, \cdot)^{1/2}$. The dual pairings on $\bV^* \times \bV$ and 
$Q^* \times Q$ are denoted by $\langle \cdot, \cdot \rangle$. Here, 
$\bV^* $ and $Q^*$ denote the duals of $\bV$ and $Q$, respectively.  We assume that $ b(\cdot, \cdot) $ is  a bilinear form  on   $\bV \times Q $,  satisfying the following   conditions.  
 \begin{equation}
\label{inf-sup_a}
\du{\inf}{p \in Q}{} \ \du {\sup} {\bv \in \bV}{} \ \frac {b(\bv, p)}{\|p\| \ 
|\bv|} =m>0, \ \ \text{and} \ \ \du{\sup}{p \in Q}{} \ \du {\sup} {\bv \in \bV}{} \ \frac {b(\bv, p)}{\|p\| \ |\bv|} =M <\infty.
\end{equation} 

For ${\bf f} \in \bV^*$, $g\in Q^*$ we consider the variational problem \eqref{abstract:variational}.
It is known that the  variational problem  \eqref{abstract:variational} has a unique solution $(\bu,p)$ 
for any ${\bf f} \in \bV^*$, $g\in Q^*$, see \cite{brenner-scott, brezzi-fortin, girault-raviart, quarteroni-valli-94, ern-guermond, B09}. 

For the SP discretization, we  let  $\bVh \subset \bV, \ \ \M_h \subset Q$ and assume that 
 \begin{equation}
\label{inf-sup-mh}
\du{\inf}{p_h \in \M_h}{} \ \du {\sup} {\bvh \in \bVh}{} \ \frac {b(\bvh, p_h)}{\|p_h\| \ 
|\bvh|} =m_h>0, 
\end{equation} 
and define the  constant $M_h$ as
\begin{equation}
\label{sup-sup-m}
M_h:=\du{\sup}{p_h \in \M_h}{} \ \du {\sup} {\bvh \in \bVh}{} \ \frac {b(\bvh, p_h)}{\|p_h\| \ 
|\bvh|} \leq M.
\end{equation}

Let the discrete operators $A_h:\bVh\to \bVh$ and $B_h:\bVh \to \M_h$ be  defined by
\[
\begin{array}{lcll}
((A_h \bu_h,\bvh))&=&a(\bu_h, \bvh)      &\Forall \bu_h,\bvh \in \bVh, \\
(B_h \bu_h, q_h )&=&  ((\bu_h, B^T q_h))= b(\bu_h, q_h)  &\Forall \bu \in \bVh, q_h \in \M_h. 
\end{array}
\]
where $((\cdot , \cdot))$ is an  inner product on  $ \bV_h \times \bV_h$, that is usually associated with a basis on $\bV_h$. 
 The discrete version of \eqref{abstract:variational} is: \\
Find $(\bu_h, p_h)  \in \bVh \times \M_h$ such that
\begin{equation}\label{eq:a-v-h}
\begin{array}{lcl}
a(\bu_h,\bvh)  +  b( \bvh, p_h) &= &(({\bf f_h},\bvh )) \ \Forall  \bvh \in \bVh,\\
b( \bu_h,q) & =& (g_h,  q_h),  \ \Forall  q_h \in \M_h, 
\end{array}
\end{equation}
where ${\bf f_h} \in \bV_h$ and $g_h \in \M_h$ are defined by 
\begin{equation}\label{eq:f-g-h}
(({\bf f_h},\bvh )) = \<{\bf f_h},\bvh \>, \ \bvh \in \bVh,  \ (g_h, q_h) = \<g_h, q_h\>, \  q_h \in \M_h.
\end{equation}
The matrix or opperatorial form of \eqref{eq:a-v-h} is: 
\begin{equation}\label{eq:OP-h}
\begin{array}{lcl}
A_h \bu_h   + B_h^T  p_h &= & {\bf f_h},\\
B_h \bu_h  & =& g_h.
\end{array}
\end{equation}
 
It is well known from  \cite{braess, brenner-scott, FJS-inf-sup, xu-zikatanov-BBtheory} that, under the assumption \eqref{inf-sup-mh},   the problem \eqref{eq:a-v-h} has a unique solution  $(\bu_h, p_h)$   and 

 \[
 |\bu -\bu_h| + \|p-p_h\| \leq  C(m_0, M)\, \left (
 \inf_{\bv_h \in \bV_h}   |\bu -\bv_h|    +   \inf_{q_h \in \M_h} \|p-q_h\| \right ), 
\]  
where $(\bu, p)$ is the   solution  of the continuous problem \eqref{abstract:variational}.

Let $S_h:\M_h \to \M_h$, be the discrete Schur complement  defined by $S_h:= B_h A_h^{-1} B_h^T$. 
It is easy to check that $S_h$ is a symmetric and positive definite operator on $\M_h$. We have that $(\cdot ,\cdot)_{S_h}:=(S_h\cdot ,\cdot)$ is another inner product on $\M_h$ with the induced normed denoted by $\|\cdot\|_{S_h}$. It is well known that  the lowest and the largest eigenvalues of $S_h$ are $m_h^2$ and $M_h^2$, respectively. Thus,
\begin{equation}\label{eq.Sh}
m_h\|q_h\| \leq \|q_h\|_{S_h}=(S_h q_h, q_h)^{1/2} \leq M_h  \|q_h\|    \ \Forall  q_h \in \M_h.
\end{equation}
\begin{remark}\label{rem:Sh} On $ \bV_h$, we  consider the same   norm as the norm on $\bV$. The  inner product  $((\cdot, \cdot))$ on $ \bV_h \times \bV_h$ is not the restriction of the inner product $a(\cdot,\cdot)$,   and is used only for  defining the discrete operators  $A_h$ and $B_h^T$. In what follows, we will need in fact only to work with  $A_h^{-1} B_h^T : \M_h \to \bV_h$ and $S_h= B_h A_h^{-1} B_h^T$  which are  independent of the choice of the  inner product  $((\cdot, \cdot))$. Indeed, if $q_h \in \M_h$ is  arbitrary,  then $\bw_h=  A_h^{-1} B_h^T q_h$  is the unique solution  of   the problem 
\[
a( \bw_h, \bv)  = b(\bv, q_h),  \Forall  \bv \in \bVh,
\]
and  $S_h q_h=B_h A_h^{-1} B_h^T q_h = B_h \bw_h $  does not depend on the  inner product $((\cdot, \cdot))$. 
We also note that if $r_h \in \M_h$ is arbitrary and  $\bv_h=  A_h^{-1} B_h^T r_h$, then 
\begin{equation}\label{eq:IPonV-h}
a(\bw_h, \bv_h)= b(\bv_h,q_h) = (B_h \ A_h^{-1} B_h^T r_h, q_h) = (S_h q_h, r_h)= (q_h,r_h)_{S_h}. 
\end{equation}
In particular, we have
\begin{equation}\label{eq:Norm-onV-h}
|\bw_h|^2= a( \bw_h, \bw_h)  =\|q_h\|^2_{S_h}. 
\end{equation}
\end{remark}

Using the Schur complement $S_h$, the  system \eqref{eq:OP-h} can be decoupled to
\begin{equation}\label{eq:OPS-h}
\begin{array}{lcl}
S_h\, p_h & =& B_h A_h^{-1} {\bf f_h}- g_h\\
 \bu_h   &= & A_h^{-1} ({\bf f_h} - B^T  p_h). 
\end{array}
\end{equation}
 
\section{Uzawa, Uzawa Gradient and Uzawa Conjugate Gradient Algorithms}\label{section:UGC}

First, we present a unified  variational form of the  Uzawa, the Uzawa gradient, and the Uzawa conjugate gradient  algorithms   for solving the SPP \eqref{eq:a-v-h}. 
The standard  U and  UG algorithms  can be rewritten such that they differ only by the way the relaxation parameter $\alpha$ is chosen. For the Uzawa algorithm, we have to choose  $\alpha=\alpha_0$    a fixed number in the interval $\left (0, \frac{2}{M_h^2} \right )$.  For the UG algorithm,  the parameter $\alpha$  is chosen to impose the orthogonality of consecutive residuals associated with the second equation in \eqref{eq:a-v-h}. 
The first step for Uzawa is identical with the first step of UG. We combine the two algorithm in:
\begin{algorithm} (U-UG) Algorithms \label{alg:U-UG}
\vspace{0.1in}

 {\bf Step 1:}  {\bf Set}  $\bu_0=0 \in \bVh$, $p_0 \in \M_h$, 
{\bf compute}  $\bu_{1} \in  \bVh $,  $q_1 \in \M_h$ by 
\[ 
 \begin{aligned}
&  a( \bu_{1}, \bv) &=& \ ((\bbf_h,\bv))  - b(\bv, p_{0}), \ \Forall  \bv \in \bVh\\
&    (q_1, q)  & =  & \ b(\bu_1 ,q) -(g_h, q) , \ \Forall   q \in \M_h.
\end{aligned}
\]
\vspace{0.1in}

{\bf  Step 2 :} {\bf For $j=1,2,\ldots, $} {\bf  compute}  $\bh_j, \alpha_j,  p_j, \bu_{j+1}, q_{j+1}$  by  
\[ 
 \begin{aligned}
&  {\bf (U-UG1)}  \ \ \ \   & a( \bh_{j}, \bv)  = &  - b(\bv, q_j), \ \bv \in \bVh\\
&  {\bf (U\alpha)}  \ \ \ \   & \alpha_j=&\alpha_0 \ \text{for the Uzawa algorithm or}\\
&  {\bf (UG\alpha)}  \ \ \ \   & \alpha_j=& - \frac{(q_j, q_j)}{b(\bh_j,q_j)} =\frac{(q_j, q_j)}{(q_j,q_j)_{S_h}}, \ \text{for the UG algorithm} \\
&  {\bf  (U-UG2)}  \ \ \ \   & p_{j}= & p_{j-1} + \alpha_j \   q_j \\
&  {\bf (U-UG3)}  \ \ \ \   & \bu_{j+1} = &  \bu_j + \alpha_j\  \bh_j \\
&  {\bf (U-UG4)}  \ \ \ \   & (q_{j+1}, q) = & \ b(\bu_{j+1} ,q) -(g_h, q)  , \ \Forall   q \in \M_h.
\end{aligned}
\]
\end{algorithm}

In the second identity in (UG$\alpha$),  we involved  Remark \ref{rem:Sh} and  (UG1). 
One can slightly modify the UG algorithm to obtain the UCG algorithm, as done in, e.g.,  \cite{braess, verfurtth84}. 

\begin{algorithm} (UCG) Algorithm \label{alg:UCG}
\vspace{0.1in}

 {\bf Step 1:}  {\bf Set}  $\bu_0=0 \in \bVh$, $p_0 \in \M_h$. {\bf Compute}  $\bu_{1} \in  \bVh $,  $q_1, d_1  \in \M_h$ by 
\[ 
 \begin{aligned}
&  a( \bu_{1}, \bv)& = & ((\bbf_h,\bv))  - b(\bv, p_{0}), \ \bv \in \bVh\\
&      (q_1, q)  & =  & \ b(\bu_1 ,q) -(g_h, q) , \ \Forall   q \in \M_h, \ \ d_1:=q_1.
\end{aligned}
\]
\vspace{0.1in}

{\bf  Step 2} {\bf For $j=1,2,\ldots, $}  {\bf  compute} 
 $\bh_j, \alpha_j,  p_j, \bu_{j+1}, q_{j+1}, \beta_j, d_{j+1}$  by  
\[ 
 \begin{aligned}
&  {\bf (UCG1)}  \ \ \ \   & a( \bh_{j}, \bv)  = &  - b(\bv, d_j), \ \bv \in \bVh\\
&  {\bf (UCG\alpha)}  \ \ \ \   & \alpha_j=& - \frac{(q_j, q_j)}{b(\bh_j,q_j)} =\frac{(q_j, q_j)}{(d_j,q_j)_{S_h}}\\
&  {\bf  (UCG2)}  \ \ \ \   & p_{j}= & p_{j-1} + \alpha_j \   d_j \\
&  {\bf (UCG3)}  \ \ \ \   & \bu_{j+1} = &  \bu_j + \alpha_j\  \bh_j \\
&  {\bf (UCG4)}  \ \ \ \   &  (q_{j+1}, q) = & \ b(\bu_{j+1} ,q) -(g_h, q)  , \ \Forall   q \in \M_h \\
& {\bf (UCG\beta)}  \ \ \ \   & \beta_j=&  \frac{(q_{j+1}, q_{j+1})}{(q_j,q_j)} \\
&  {\bf (UCG6)}  \ \ \ \   & d_{j+1}= & q_{j+1} +\beta_j d_j \\
\end{aligned}
\]
\end{algorithm}

\begin{remark}\label{rem:CG4Sh}
It is not difficult to check that the UG and UCG algorithms   produce  the standard gradient and the standard conjugate gradient algorithms for solving the first equation in \eqref{eq:OPS-h}. 

\end{remark} 
\begin{theorem}\label{th:normUp}  Let  $(\bu_{h},p_h)$ be the solution of  
 \eqref{eq:a-v-h}, and let  $\{(\bu_{j+1},p_j)\}_{j \geq 0}$ be the iterations produced by a U, UG, or UCG algorithm. Then, for  $ j \geq 0$, 
\begin{equation}\label{eq:umuj-h}
\bu_{j+1} -\bu_h  =   A_h^{-1} B_h^T(p_h -p_j), 
\end{equation}

\begin{equation}\label{eq:qerrp-h}
q_{j+1}= S_h(p_h-p_{j}), 
\end{equation}
and consequently, for  $ j \geq 1$, 
\begin{equation}\label{eq:qEstimator-ph}
\frac{1} {M_h^2}\,  \|q_j\| \leq \|p_{j-1} -p_h \| \leq  \frac{1} {m_h^2}\,  \|q_j\|. 
\end{equation}

\begin{equation}\label{eq:qEstimator-uh}
\frac{m_h} {M_h^2}\,  \|q_j\| \leq |\bu_{j} -\bu_h | \leq  \frac{M_h} {m_h^2}\,  \|q_j\|,
\end{equation}

\end{theorem}
\begin{proof} By induction over $j$, it is easy to prove (for any of the U, UG, or UCG) that  
\begin{equation}\label{eq:uj-h}
a(\bu_{j+1},\bv) + b(\bv, p_{j}) =  ((\bbf_h,\bv)), \ \Forall  \ \bv \in \bV_h. 
\end{equation}

Combining  the first equation in   \eqref{eq:a-v-h} and  \eqref{eq:uj-h},  we get
\[
a( \bu_{j+1} -\bu_h, \bv)  = b(\bv, p_h -p_j),  \Forall  \bv \in \bVh, 
\]
which  gives  \eqref{eq:umuj-h}.  
From  {\bf (U4), (UG4)}, or {\bf (UCG4)}, the second equation of  \eqref{eq:OP-h}, and \eqref{eq:umuj-h} we get 
\[
q_{j+1} =B_h \bu_{j+1}-  g_h = B_h (\bu_{j+1} -\bu_h)  =S_h(p_h -p_{j}). 
\]
which proves \eqref{eq:qerrp-h}.  As a consequence of  \ref{eq:umuj-h}, the estimate \eqref{eq.Sh},  and Remark \ref{rem:Sh}, for $j \geq 1$, we have
\begin{equation}\label{eq:numuj-hh}
m_h  \|p_h -p_{j-1}\| \leq |\bu_{j} -\bu_h | =   \|p_h -p_{j-1}\|_{S_h} \leq  M_h  \|p_h -p_{j-1}\|.   
\end{equation}
Using \eqref{eq:qerrp-h} and the fact that $m_h^2$ and $M_h^2$ are  the extreme eigenvalues of $S_h$, we get 
\begin{equation}\label{eq:normq-h}
m_h^2  \|p_h -p_{j-1}\| \leq   \|S_h(p_h -p_{j-1})\| =\|q_{j} \|  \leq  {M_h^2}    \|p_h -p_{j-1}\|.
\end{equation}
The estimates  \eqref{eq:qEstimator-ph}  and  \eqref{eq:qEstimator-uh} are a direct consequence of \eqref{eq:numuj-hh}
and \eqref{eq:normq-h}.
\end{proof}
As a consequence of Theorem \ref{th:normUp}, we obtain
\begin{equation}\label{eq:erup-h}
\frac{1+ m_h} {M_h^2}\,  \|q_{j}\| \leq  |\bu_h-\bu_j|  +   \|p_h-p_{j-1}\|  \leq  \frac{1+M_h} {m_h^2}  \|q_{j}\|,  
\end{equation} 
which says  that $\|q_{j}\|$ is  an  estimator   for the global iteration error providing good  upper and lower bounds.  

\begin{theorem}\label{th:U-UG-UCG}
 Let  $(\bu_{h},p_h)$ be the solution of   \eqref{eq:a-v-h}, and let  $\{(\bu_{j+1},p_j)\}_{j \geq 0}$ be the iterations produced by a U, UG or UCG algorithm. Then, \\
   $(\bu_{j+1},p_j) \to (\bu_{h}, p_h)$, and consequently $q_j \to 0$. 
\end{theorem}
\begin{proof} 
For the Uzawa algorithm,  it is easy to check that 
\begin{equation}\label{eq:errp-h}
p_h-p_j= (I -\alpha S_h) (p_h-p_{j-1}). 
\end{equation}
Using  that the  eigenvalues of the symmetric operator $S_h$ are $m_h^2$ and $M_h^2$, we have  
\begin{equation}\label{eq:rateU}
\| I -\alpha S_h\| =\max\{|1-\alpha m_h^2 |, |1-\alpha M_h^2 |\} <1, \  \text{for} \  \alpha \in \left (0, \frac {2} {M_h^2}\right ). 
\end{equation}
Thus, $p_j \to p_h$.  For the UG and UCG, by using   Remark \ref{rem:CG4Sh}, the following estimates  is  well known from \cite{greub-rheinboldt59}, \cite{BacShu13} and others. 
\begin{equation}\label{eq:pconv-h}
 \|p_h-p_j\|_{S_h}  \leq  \frac{{M_h}^2-m_h^2} {{M_h}^2 + m_h^2}\  \|p_h-p_{j-1}\|_{S_h}.  
\end{equation}
The estimate gives $p_j \to p_h$. Using \eqref{eq:umuj-h}, we also get that $\bu_{j+1} \to \bu_h$ for the  U, UG, or UCG algorithm. The fact that $q_j \to 0$  follows from \eqref{eq:qEstimator-ph}. 
\end{proof}
%
\section{Cascadic Algorithm for Saddle Point Problems}\label{section:CM}

In this section we will define concrete (CM) algorithms by specifying   the {\it process} of {\bf Step CM2} and by 
defining a level change condition (LC). We will use the notation and the setting of the previous sections.
Assume that  we can easily build a  sequence of pairs $\{({\bV_k}, {\M_k})\}_{k \geq 1}$  that satisfies  a discrete $\inf-\sup$ condition for every $k\geq 1$, and that $h_k$ is a mesh parameter associated with the pair $(\bV_k, \M_k)$ such that $h_k \to 0$. We define 
\begin{equation}
\label{inf-sup-m-k}
m_k:=\du{\inf}{p_k \in \M_k}{} \ \du {\sup} {\bv_k \in \bV_k}{} \ \frac {b(\bv_k, p_k)}{\|p_k\| \ |\bv_k|}. 
\end{equation} 
and 
\begin{equation}\label{sup-sup-Mk}
M_k:=\du{\sup}{p_k \in \M_k}{} \ \du {\sup} {\bv_k \in \bV_k}{} \ \frac {b(\bv_k, p_k)}{\|p_k\| \ 
|\bv_k|} \leq M.
\end{equation} 

  In order to prove the convergence  of  Algorithm \ref{alg:CM}  we further introduce the following  assumptions:
 \begin{itemize}
 
\item [{$\bf (A_1)$}] The family  $\{(\bV_k, \M_k)\}_{k\geq 1}$ is stable: \\
 There exists $m_0>0$ such that $m_k\geq m_0$, for $k=1,2,\cdots$.
\vspace{0.1in}

\item  [{$\bf (A_2)$}] The  {\it process}  of  {\bf Step CM2}  is defined by  {\bf Step 2} of U, UG or UCG algorithm. 
In the Uzawa solver case, we take $\alpha_0 \in \left (0, 2/M^2\right )$. 
\vspace{0.1in}

\item [{$\bf (A_3)$}]   If $(\bu, p)$ is the  solution  \eqref{abstract:variational}, and $(\bu^{(k)}, p^{(k)})$ the solution of \eqref{eq:a-v-h} on $({\bV_k}, {\M_k})$, then there exist 
$C_0 =C_0(\bu,p)$ and $s>0$  independent of $k$, such that
 \begin{equation}\label{eq:appox-k}
|\bu-\bu^{(k)}| +\|p-p^{(k)} \| \le C_0\,  h_k ^s.
\end{equation}
\vspace{0.1in}

\item [{$\bf (A_4)$}]  The {\it level change} condition is 
 \begin{equation*}\label{eq:stopC}
\text{(LC)} \ \ \ \ \ \ \ \  \ \|q_{j+1}\| \leq  C_{lc}\  h_k^s, 
\end{equation*}
where $C_{lc}$ is a constant independent of $k$.
\end{itemize}
 We further consider that  a sequence of  prolongation operators \\  $ \P_{k,k+1} : \M_k \to \M_{k+1}$ is available.  We are  ready now to state our main result:

\begin{theorem}\label{th:CMconvergence}
\bigskip
Assume that  {$\bf (A_1)-(A_4)$}  are satisfied. If $(\bu_{j+1},p_{j})$ is the last iteration computed by the CM algorithm  on $(\bV_k, \M_k)$, then there exists a constant $C$ depending only on $m_0, M, C_0, C_{lc} $,  and $\alpha_0$ in 
the Uzawa level solver case, such that 
 \begin{equation}\label{eq:CONV-j}
  |\bu -\bu_{j+1}|  +  \|p-p_{j}\|  \le C h_k^s.
\end{equation}

\end{theorem}

\begin{proof} From \eqref{eq:appox-k}  and \eqref{sup-sup-Mk}, we have that $m_0\leq m_k \leq M_k \leq M$, for $k=1,2,\cdots $. Thus, from Theorem \ref{th:normUp} or the equation  \eqref{eq:erup-h}  we get that 
\[
|\bu_{j+1}-\bu^{(k)}| +\|p_{j}-p^{(k)} \| \le C_1  \|q_{j+1}\|,
\]
with $C_1$ depending only on $m_0$ and $M$ (and $\alpha_0$ in the U-case). 
If $(\bu_{j+1},p_{j})$ is the last iteration computed by the CM algorithm  on $(\bV_k, \M_k)$, then  by assumption  {$\bf (A_4)$}, we have $ \|q_{j+1}\| \leq    C_{lc}\, h_k^s$, and consequently,  

 \begin{equation}\label{eq:appox-k}
|\bu_{j+1}-\bu^{(k)}| +\|p_{j}-p^{(k)} \| \le  C_1  C_{lc}\,  h_k ^s.
\end{equation}
We note here that, due to Theorem \ref{th:U-UG-UCG},  if  infinitely many iterates would be performed on  a fixed level, then we had  that $ \|q_{j+1}\| \to 0$, which contradicts the level change assumption {$\bf (A_4)$}. Consequently,  on each level  the algorithms perform a finite number of iterations.
The convergence estimate \eqref{eq:CONV-j} is a direct consequence of {$\bf (A_3)$}, \eqref{eq:appox-k}, and the triangle inequality.  
\end{proof}

\begin{remark} We did not use any assumption on the  prolongation operators  $ \P_{k,k+1} : \M_k  \to \M_{k+1}$. Nevertheless, from Theorem \ref{th:Uconv} and Theorem \ref{th:UG-UCG}, under the stability assumption {$\bf (A_1)$},  we can conclude that  on each fixed level $k$  the reduction error for consecutive steps can be bounded by a factor $\rho \in(0,1)$  independent of $k$. Thus, under some natural assumption on  the prolongation $ \P_{k,k+1}$ 
such us
\[
\| \P_{k,k+1}\,  p -p\| \leq C_2\,  h_k ^s, \ \Forall p\in \M_k,
\]
it is easy to prove that  the CM algorithm converges with 
the approximation order of \eqref{eq:CONV-j}, and in addition, the number of iterations on each level is bounded by a fixed number   $N_{maxit}$  independent of level $k$. 
\end{remark}

 Thus, if the number of iteration (or the amount of work) of the  CM algorithm is associated with a  water cascade flow with  the steps corresponding to our multilevel spaces, we can claim that the flow { does not spread out}. This makes our proposed algorithm a {cascadic non-spreading} iteration process.   
The {\it non-spreading cascadic}  phenomena can be also ``watched'' on  the last column of  Table \ref{table:1}-Table \ref{table:4}.   In what follows,  the  {\bf CM} algorithm  with {\bf  U, UG}, or  {\bf UCG}  as level solver  defines  the corresponding  {\bf CMU, CMUG}, or {\bf CMUCG} algorithm. 

\section{Numerical results for the Stokes system} \label{section:NR-Stokes}

In this section, we show the numerical performance of the CM algorithm, emphasizing on the way one should 
choose the level change criterion once information about the order of the discretization error is available. 
We implemented the {\bf CMU, CMUG}, and {\bf CMUCG} algorithms  for the  discretization of  the standard Stokes system. For each level $k$  we record the errors   $|\bu -\bu_{j+1}| $  and   $\|p-p_{j}\|$ where  $(\bu_{j+1},p_{j})$ is the last iteration computed on $(\bV_k, \M_k)$. 

First, we considered $\Omega$ to be  the unit square $(0, 1)^2$ and defined the data for the Stokes system, such that  the exact solution is $p=2/3-x^2 -y^2$ and  $u_1=u_2 =1/{2\pi^2}\,  \sin(\pi x) \sin(\pi y)$. We discretize using   two  known stable families of pairs: $P_2-P_0$,  and ($P_2-P_1$) -Taylor-Hood (T-H). To construct the  spaces  $(\bV_k, \M_k )$,  we  define  the original triangulation $\T_1$ on $\Omega$  given by  the Union Jack pattern. The family of uniform meshes $\{\T_k\}_{k \geq 1}$ is defined by a uniform refinement strategy, i.e., $\T_{k+1}$ is obtained from $\T_{k} $ by splitting each triangle of   $\T_{k} $ in four similar triangles.  
Since the sequence of spaces  $\{ \M_k \}$ is nested, in both   $P_2-P_0$ and T-H discretizations, the  prolongation operators  $ \P_{k,k+1} : \M_k \to \M_{k+1}$  are simply the embedding operators. 

 For  the   $P_2-P_0$ element,  the discretization error is  $O(h) $, so we use the  
{\it  level change  condition:  (LC) } $  \|q_{j+1}\| \leq \frac 1{16}  {h_k}$. For comparison between {\bf CMU, CMUG}, and {\bf CMUCG}, see Table \ref{table:1}. 

 \begin{center}
\begin{table} [!ht]
\begin{tabular}{|l|c|c|c|c|c|} 
\hline  
       {\bf CMU, $\alpha=0.8$}& {$|\bu-\bu_{j+1}|$}& {Rate} & {$\|p-p_j\|$}& {Rate}  & {\# of iter}\\ 
\hline         
         k=4    & 0.0384038 & & Ê0.0434820 &  & 16 \\        
  \hline     
         k=5    &0.0200707 & 0.95 &  0.0264921 &0.72 & 8 \\
 \hline     
        k=6    & 0.0103764 &0.97 &0.0156126 & 0.76 & 10   \\
  \hline
         k=7   &0.0053116 &0.96 & 0.0086539 &0.85& 11 \\      
 \hline
         k=8   &{0.0026943} &{0.98}&{0.0045779} &{0.92}& 11 \\       
\hline     
       {\bf CMUG } & {}& {} & {}  & {} & {} \\
\hline         
         k=4    &0.0386718 & &0.0467490&& 13\\        
  \hline     
         k=5    &0.0201762 &0.94 & 0.027242 &0.78 &6 \\
 \hline     
        k=6    &Ê0.0103487 & 0.96 & 0.0150688 &0.86 & 6  \\
  \hline
         k=7   &0.0052530 &0.98  &0.0079920 &0.91& 5\\      
 \hline
         k=8   &{0.0026532} &{0.99} & {0.0041757}& {0.94}& 5 \\        
\hline
       {\bf CMUCG } & {}& {} & {}  & {} & {} \\
\hline         
         k=4     & 0.0415028  &  & 0.0732229  & & 9\\        
  \hline     
         k=5    &0.0213266 &0.96&  0.0373755 &0.97& 3 \\
 \hline     
        k=6     &0.0107704  &0.98 & 0.0188775 & 0.98 & 4  \\
  \hline
         k=7   &0.0054313 & 0.99 &0.0047840 &0.99 & 3 \\      
\hline
         k=8   &{0.0027269} &{0.99}&{0.0045623}&{0.99}& 3\\      
\hline
\end{tabular}
 \caption{{\bf CM} $P_2-P_0$ discretization  with LC:\  $  \|q_{j+1}\| < 1/16 \ h_k$}
 \label{table:1}
\end{table}
\end{center}

  \begin{center}
\begin{table} [!ht]
\begin{tabular}{|l|c|c|c|c|c|} 
\hline  
       {\bf CMU, $\alpha=1$}& {$|\bu-\bu_{j+1}|$}& {Rate} & {$\|p-p_j\|$}& {Rate}  & {\# of iter}\\ 
\hline         
         k=4    & 0.0008450 & & 0.0009184 &  & 23\\        
  \hline     
         k=5    &0.0002081& 2.02 &  0.0002208 &2.05 & 6 \\
 \hline      
        k=6    &0.0000517 &1.83 &0.0000546 & 2.02 & 6   \\
  \hline
         k=7   & 0.0000129 &2.00 &0.0000138 &1.98& 6 \\      
 \hline
         k=8   &{0.0000032} &{1.98}&{0.0000035} &{1.96}& 6 \\       
\hline     
       {\bf CMUG } & {}& {} & {}  & {} & {} \\
\hline         
         k=4    & 0.0007491  & &0.0006847&& 14\\        
  \hline     
         k=5    & 0.0001770  &2.08 & 0.0001121 &2.61 &4 \\
 \hline     
        k=6    & 0.0000442 & 2.00 &0.0000278 &2.01 & 2  \\
  \hline
         k=7   &0.0000110 &2.00  &0.0000071 &1.97& 2\\      
 \hline
         k=8   &{0.0000027} &{2.00} & {0.0000018}& {1.97}& 2 \\        
\hline
       {\bf CMUCG } & {}& {} & {}  & {} & {} \\
\hline         
         k=4     & 0.0008260  &  & 0.00087068  & & 7\\        
  \hline     
         k=5    &0.0001757 &2.23& 0.0001054 &3.04& 2 \\
 \hline     
        k=6     & 0.0000438  &2.00 & 0.0000259 & 2.02 & 2  \\
  \hline
         k=7   & 0.0000109 & 2.00 & 0.0000065 &2.00 & 2 \\      
\hline
         k=8   &{0.0000027} &{2.00}&{0.0000016}&{1.99}& 2\\      
\hline
\end{tabular}
 \caption{{\bf CM}  T-H discretization  with LC: $  \|q_{j+1}\| < 1/16 \ h_k^2$}
 \label{table:2}
\end{table}
\end{center}
For  the   {Taylor-Hood} element,  the discretization error is  $O(h^2) $, so we use the  
{\it level change  condition:  (LC) } $  \|q_{j+1}\|  \leq \frac 1{16} h_k^2$. 
For a comparison between {\bf CMU, CMUG}, and {\bf CMUCG}, see Table \ref{table:2}. When we impose the iteration  only on the last level ($k=8$) using the same stopping criterion,  we obtain similar errors, by using  $45, 29$, and $20$ iterations for  {\bf CMU, CMUG}, and {\bf CMUCG}, respectively.  For the convex case, it seems that   {\bf CMUCG} has only a slightly better performance when compared with  {\bf CMUG}. Nevertheless, the advantage of  {\bf CMUCG}  is more significant  in the non-convex case. 

Secondly, we performed  numerical experiments  for  the Stokes system on the  $L$-shaped domain $\Omega:=(-1, 1)^2 \setminus [0, 1]\times [-1, 0] $ using the ($P_2-P_1$)-T-H discretization. We chose the data such that the exact solution is  $u_1=u_2=r^{2/3}\sin(\frac{2}{3}\theta) (1-x^2)(1-y^2)$, and  $p=2/3-x^2 -y^2$.  Note that $u_1 \notin H^{1+ 2/3}$.  
We used  {\it quasi-uniform} meshes  and  {\it graded meshes} also. For both types of  refinement, we started with the initial triangulation $\T_1$  being the Union Jack pattern on each of the three unit squares of the domain. For the uniform refinement case, the family of quasi-uniform meshes $\{\T_k\}_{k \geq 1}$ is defined  by a uniform refinement strategy as in the convex case. For the graded meshes  $\T_{k+1}$ is obtained from $\T_{k} $ by splitting each triangle of   $\T_{k} $ in four  triangles  as follows: we refine   by  dividing all  the edges  that contain the singular point $(0, 0)$  under a fixed ratio $\kappa>0$ such that the segment containing the singular point is $\kappa$-times the other segment, (see e.g., \cite{BNZ1,BNZ2}).  
We used   $N_k=N_{d.o.f}$  as the complexity measure on  $(\bV_k,  \M_k)$, where $N_{d.o.f}$ is  the number of degrees of freedom associated with a scalar discrete Laplacian on $\bV_k$. For the uniform refinement,  the discretization error is $O(N_k^{-1/3}) $, so we used the  
{\it  level change  condition:  (LC) } $  \|q_{j+1}\| \leq \frac 1{8}  {N_k^{-1/3}}$. The performances of  {\bf CMUG}, and {\bf CMUCG} are similar, see Table \ref{table:3}.
For  the graded meshes refinement, we experimented with various values of $\kappa<1$ in order  to approach  the optimal order of convergence $O(N_k^{-1})$ exhibited in the convex case, and used the {\it stopping criterion}: 
$  \|q_{j+1}\| \le \frac{1} {8} \, N_k^{-1} $. For the comparison between {\bf CMUG}, and {\bf CMUCG}, see Table \ref{table:4}.

  \begin{center}
\begin{table} [!ht]
\begin{tabular}{|l|c|c|c|c|c|} 
\hline  
       {\bf CMUG, $\kappa=1$}& {$|\bu-\bu_{j+1}|$}& {Rate} & {$\|p-p_j\|$}& {Rate}  & {\# of iter}\\ 
\hline         
         k=4    & 0.0509160 & &0.0157147 &  & 7\\        
  \hline     
         k=5    &0.0320362& 0.67 &  0.0093171 &0.75 & 2 \\
 \hline      
        k=6    &0.020178 & 0.67 & 0.0060782 & 0.62 & 2  \\
  \hline
         k=7   & 0.0127138 & 0.67&0.0038458 &0.66& 2 \\      
 \hline
         k=8   &{0.0080074} &{0.67}&{0.0023277} &{0.72}& 2\\       
\hline     
      {\bf CMUCG, $\kappa=1$}  & {}& {} & {}  & {} & {} \\
\hline         
         k=4    & 0.0518455 & &0.0219035 && 4\\        
  \hline     
         k=5    &0.0320592  &0.69 &0.0092441  &1.24 & 2 \\
 \hline     
        k=6    &0.0127075 & 0.67 & 0.0056854 &0.70 & 2  \\
  \hline
         k=7   & 0.0126586 &0.67  & 0.0036609 &0.64 & 2 \\      
 \hline
         k=8   &{0.0080042} &{0.67} & {0.0022305}& {0.71}& 2 \\        
\hline
\end{tabular}
 \caption{{\bf CM}  T-H on uniform refinement, LC: $  \|q_{j+1}\| < 1/8 \ N^{-1/3}_{d.o.f}$}
 \label{table:3}
\end{table}
\end{center}

  \begin{center}
\begin{table} [!ht]
\begin{tabular}{|l|c|c|c|c|c|} 
\hline  
       {\bf CMUG, $\kappa=1/8$}& {$|\bu-\bu_{j+1}|$}& {Rate} & {$\|p-p_j\|$}& {Rate}  & {\# of iter}\\      
\hline         
         k=4    & 0.0185020 & &0.0052936 && 13\\        
  \hline     
         k=5    & 0.0062828  &1.56 &Ê0.0015040 &1.81 & 7 \\
 \hline     
        k=6    &0.0019275 & 1.70 & 0.0003630 &2.05 & 9  \\
  \hline
         k=7   &  0.0005533 &1.80  &Ê0.0000852 &2.09 & 9 \\      
 \hline
         k=8   &{0.0001524} &{1.86} & {0.0000204}& {2.06}& 8 \\  
\hline
         k=9   &{0.0000409} &{1.90} & {0.0000048}& {2.07}& 8 \\            
\hline
       {\bf CMUCG, $\kappa=1/8$} & {}& {} & {}  & {} & {} \\
\hline         
         k=4     & 0.0185008 &  & 0.0052371 & & 7\\        
  \hline     
         k=5    &0.0062794 &1.56 &0.0014294 &1.87& 4 \\
 \hline     
        k=6     & 0.0019272  &1.70 &0.0003316  & 2.10 & 5  \\
  \hline
         k=7   & 0.0005533 & 1.80 & 0.0000822 &2.01 & 3 \\      
\hline
         k=8   &{0.0001524} &{1.86}&{0.0000191}&{2.10}& 4\\   
\hline
         k=9   &{Ê0.0000409} &{1.90}&{0.0000045}&{2.15}& 6\\                
\hline
\end{tabular}
 \caption{{\bf CM}  T-H  on graded meshes   with LC: $  \|q_{j+1}\| < 1/8 \ N^{-1}_{d.o.f}$}
 \label{table:4}
\end{table}
\end{center}

We note that  by choosing  graded meshes and an appropriate level change condition,  we can  improve  the rate of convergence of the CM algorithm. Both CMUG and  CMUCG  recover  a   {\it better than expected}  rate of convergence  for the pressure  and a {\it close to optimal} rate of convergence for the  velocity. The {\it optimal} choice  of the running parameters  $C_{lc}$ and $\kappa$, together with  finding quasi-optimal approximation spaces for Stokes and other SPPs on polygonal or polyhedral domains  are challenging problems that will be further  investigated. 
  
\section{Conclusion} \label{section:conclusion} 

 We presented cascadic type  algorithms for discretizing  saddle point problems for the particular case  when the form $a(\cdot, \cdot)$ is symmetric and coercive. The new algorithms  are based on  the existence of multilevel sequences of nested approximation spaces that are stable.  
  We focused on cascadic multilevel algorithms of Schur complement type  with Uzawa, 
Uzawa gradient and the Uzawa conjugate gradient as level solvers.
 The level change criterion requires that  the iteration error be close to the expected discretization error, and we enforced it  by using  an efficient and  easy to compute  residual estimators for  the  iteration error. The theoretical results and experiments show that the algorithms can achieve optimal or  close to optimal  approximation rates by performing a  non-increasing  number of  iterations on each level. 
The  main computational challenge for each  iteration  is inverting  operators  of discrete  Laplacian type. If we  efficiently  invert these operators,  we obtain a  significant  reduction of the overall running time as we compare our CM  algorithm with other  non-multilevel  iterative methods. The algorithms can be applied to a large class of first order systems of  PDEs that can be reformulated at the continuous level as  symmetric SPPs, such as 
 the  $\div$-$\curl$ system and the  Maxwell equations (see  \cite{BM12,BPls03, BKPls05}).

\thanks{The  author would like to thank Lu Shu for the help with some of the numerical experiments.}

\begin{thebibliography}{10}

\bibitem{B06}
C.~Bacuta.
\newblock A unified approach for {U}zawa algorithms.
\newblock {\em SIAM J. Numer. Anal.}, 44(6):2633--2649, 2006.

\bibitem{B09}
C.~Bacuta.
\newblock Schur complements on {H}ilbert spaces and saddle point systems.
\newblock {\em J. Comput. Appl. Math.}, 225(2):581--593, 2009.

\bibitem{BM12}
C.~Bacuta and P.~Monk.
\newblock Multilevel discretization of symmetric saddle point systems without
  the discrete {LBB} condition.
\newblock {\em Appl. Numer. Math.}, 62(6):667--681, 2012.

\bibitem{BNZ1}
C.~Bacuta, V.~Nistor, and L.~Zikatanov.
\newblock Improving the rate of convergence of `high order finite elements' on
  polygons and domains with cups.
\newblock {\em Numerische Mathematik}, 100(2):165 --184, 2005.

\bibitem{BNZ2}
C.~Bacuta, V.~Nistor, and L.~Zikatanov.
\newblock Improving the rate of convergence of `high order finite elements' on
  polyhedra i: apriori estimates.
\newblock {\em Numerical Functional Analysis and Optimization}, 26(6):613 --
  639, 2005.

\bibitem{BacShu13}
C.~Bacuta and L.~Shu.
\newblock Multilevel gradient uzawa algorithms for symmetric saddle point
  problems.
\newblock {\em Jornal of Scientific Computing}, 2013.

\bibitem{bansch-morin-nochetto}
E.~Bansch, P.~Morin, and R.H. Nocheto.
\newblock An adaptive {U}zawa fem for the {S}tokes problem:convergence without
  the inf-sup condition.
\newblock {\em SIAM J. Numer. Anal.}, 40:1027--1229, 2002.

\bibitem{benzi-golub-liesen}
M.~Benzi, G.~Golub, and J.~Liesen.
\newblock Numerical solutions of saddle point problems.
\newblock {\em Acta Numerica}, pages 1--137, 2005.

\bibitem{BornemannDeuflhard96}
F.A. Bornemann and P.~Deuflhard.
\newblock The cascadic multigrid method for elliptic problems.
\newblock {\em Numerische Mathematik}, 75:135--152, 1996.

\bibitem{braess}
D.~Braess.
\newblock {\em Finite Elements. Theory, Fast Solvers, and Applications in Solid
  Mechanics}.
\newblock Cambridge University Press, Cambridge, 1997.

\bibitem{BraessDahmen99}
D.~Braess and W.~Dahmen.
\newblock A cascadic multigrid algorithm for the {S}tokes equations.
\newblock {\em Numer. Math.}, 82(2):179--191, 1999.

\bibitem{BraessSarazin95}
D.~Braess and R.~Sarazin.
\newblock An efficient smoother for the {S}tokes problem.
\newblock {\em Appl. Numer. Math.}, 23(1):3--19, 1997.
\newblock Multilevel methods (Oberwolfach, 1995).

\bibitem{BPls03}
J.~H. Bramble and J.~Pasciak.
\newblock A new approximation technique for div-curl systems.
\newblock {\em Math. Comp.}, 73:1739--1762, 2004.

\bibitem{uzawa3}
J.~H. Bramble, J.~E. Pasciak, and A.~Vassilev.
\newblock Analysis of the inexact {U}zawa algorithm for saddle point problems.
\newblock {\em SIAM J. Numer. Anal.}, 34(3):1072--1092, 1997.

\bibitem{uzawa1}
J.~H. Bramble, J.~E. Pasciak, and A.~Vassilev.
\newblock Uzawa type algorithms for nonsymmetric saddle point problems.
\newblock {\em Math. Comp.}, 69(230):667--689, 2000.

\bibitem{bpvscale}
J.~H. Bramble, J.~E. Pasciak, and P.~S. Vassilev.
\newblock Computational scales of {S}obolev norms with application to
  preconditioning.
\newblock {\em Math. Comp.}, 69(230):463--480, 2000.

\bibitem{BrambleZhang}
J.~H. Bramble and X.~Zhang.
\newblock The analysis of multigrid methods.
\newblock In {\em Handbook of numerical analysis, Vol. VII}, pages 173--415.
  North-Holland, Amsterdam, 2000.

\bibitem{brenner-scott}
S.~Brenner and L.R. Scott.
\newblock {\em The Mathematical Theory of Finite Element Methods}.
\newblock Springer-Verlag, New York, 1994.

\bibitem{brezzi-fortin}
F.~Brezzi and M.~Fortin.
\newblock {\em Mixed and Hybrid Finite Element Methods}.
\newblock Springer-Verlag, New York, 1991.

\bibitem{dahlke-dahmen-urban}
S.~Dahlke, W.~Dahmen, and K.~Urban.
\newblock Adaptive wavelet methods for saddle point problems-optimal
  convergence rates.
\newblock {\em SIAM J. Numer. Anal.}, 40:1230--1262, 2002.

\bibitem{elman-golub}
H.C. Elman and G.~Golub.
\newblock Inexact and preconditioned {U}zawa algorithms for saddle point
  problems.
\newblock {\em SIAM J. Numer. Anal.}, 31:1645--1661, 1994.

\bibitem{girault-raviart}
V.~Girault and P.A. Raviart.
\newblock {\em Finite Element Methods for {N}avier-{S}tokes Equations},
  volume~15.
\newblock Springer-Verlag, Berlin, 1986.

\bibitem{greub-rheinboldt59}
Werner Greub and Werner Rheinboldt.
\newblock On a generalization of an inequality of {L}. {V}. {K}antorovich.
\newblock {\em Proc. Amer. Math. Soc.}, 10:407--415, 1959.

\bibitem{ern-guermond}
A.~Ern J-L. Guermond.
\newblock {\em Theory and Practice of Finite Elements}.
\newblock Springer-Verlag, New York, 2004.

\bibitem{kondratyuk-stevenson}
Y.~Kondratyuk and R.~Stevenson.
\newblock An optimal adaptive finite element method for the {S}tokes problem.
\newblock {\em SIAM J. Numer. Anal.}, 46:746--775, 2008.

\bibitem{DeuflhardLeinenYserentant}
P.~Leinen P.~Deuflhard and H.~Yserentant.
\newblock Concepts of an adaptive hierarchical finite element code.
\newblock {\em IMPACT Comput. Sci. Eng.}, 1:3--35, 1989.

\bibitem{uzawa2}
J.~E. Pasciak and A.~Vassilev.
\newblock Inexact {U}zawa algorithms for symmetric and nonsymmetric
  saddle-point problems.
\newblock In {\em Large-scale scientific computations of engineering and
  environmental problems (Varna, 1997)}, pages 91--97. Vieweg, Braunschweig,
  1998.

\bibitem{BKPls05}
J.~H.~Bramble~J. Pasciak and T.~Kolev.
\newblock A least-squares method for the time-harmonic maxwell equations.
\newblock {\em J. Numer. Math.}, 13:237--320, 2005.

\bibitem{quarteroni-valli-94}
A.~Quarteroni and A.~Valli.
\newblock {\em Numerical Approximation of Partial Differential Equations}.
\newblock Springer, Berlin, 1994.

\bibitem{FJS-inf-sup}
F.J. Sayas.
\newblock Infimum-supremum.
\newblock {\em Bol. Soc. Esp. Mat. Apl. S$\vec{\rm e}$MA}, (41):19--40, 2007.

\bibitem{Shaidurov96}
V.~Shaidurov.
\newblock Some estimates of the rate of convergence for the cascadic
  conjugate-gradient method.
\newblock {\em Computers \& Mathematics with Applications}, 31(4-5):161--171,
  1996.

\bibitem{vassilevski-wang}
P.~S. Vassilevski and J.~Wang.
\newblock Wavelet-like methods in the design of efficient multilevel
  preconditioners for elliptic pdes.
\newblock In {\em Multiresolution Analysis and Wavelets for the Numerical
  Solution of Partial Differential Equations}. Academic Press, New York, 1997.

\bibitem{verfurtth84}
R.~Verf{\"u}rth.
\newblock A combined conjugate gradient-multigrid algorithm for the numerical
  solution of the {S}tokes problem.
\newblock {\em IMA J. Numer. Anal.}, 4(4):441--455, 1984.

\bibitem{XuThesis}
J.~Xu.
\newblock {\em Theory of Multilevel Methods}.
\newblock PhD thesis, Cornell University, Ithaca, NY, 1989.
\newblock AM report 48, Dept. of Math., Penn. State Univ., University Park, PA.

\bibitem{XuSiamReview}
J.~Xu.
\newblock Iterative methods by space decomposition and subspace correction.
\newblock {\em SIAM Review}, 34:581--613, 1992.

\bibitem{xu-zikatanov-BBtheory}
Jinchao Xu and Ludmil Zikatanov.
\newblock Some observations on {B}abu\v ska and {B}rezzi theories.
\newblock {\em Numer. Math.}, 94(1):195--202, 2003.
\end{thebibliography}

\end{document}